\documentclass[12pt,reqno]{amsart}

\usepackage[T2A]{fontenc}
\usepackage{mathtext}
\usepackage[cp1251]{inputenc}
\usepackage[english]{babel}
\usepackage{amsmath,amsxtra,amssymb,eufrak}
\usepackage[all]{xy}
\usepackage{mathrsfs}
\usepackage{paralist}
\usepackage[active]{srcltx} 

\newtheorem{thm}{Theorem}[section]
\newtheorem{lm}[thm]{Lemma}
\newtheorem{co}[thm]{Corollary}
\newtheorem{pr}[thm]{Proposition}

\theoremstyle{definition}

\newtheorem{rem}[thm]{Remark}
\newtheorem{rems}[thm]{Remarks}

\numberwithin{equation}{section}

\title{Holomorphic reflexivity for locally finite and profinite groups: the Abelian and general cases}
\author{O.\,Yu.~Aristov}
\email{aristovoyu@inbox.ru}

\begin{document}

\begin{abstract}
Akbarov's theory of holomorphic reflexivity for topological Hopf algebras has been
developed in two directions, namely, by the complication of definitions when expanding the scope
and by their simplification when restricting. In the framework of the latter approach, we establish the
holomorphic reflexivity for topological Hopf algebras associated with locally finite countable groups
and second-countable profinite groups. In the Abelian case, the reflexivity is described in terms close
to the classical ones.
\end{abstract}

\maketitle
\markright{Holomorphic reflexivity}

In the paper \cite{Ak08}, Akbarov presented a holomorphic version of the Pontryagin duality and showed that
an analogue of the classical result on reflexivity is valid for compactly generated Abelian Stein groups.
He also expanded the context by considering the holomorphic duality for some class of topological Hopf
algebras, which includes algebras of holomorphic functions on general (not only Abelian) compactly
generated Stein groups, and also convolution algebras of exponential analytic functionals. The fact that
there is no need to either prove or assume the existence of the Haar measure or its noncommutative
analogues is an accomplishment of this theory. However, there are obstacles to reflexivity (for example,
in some form, the linearity of a group is a necessary condition; see \cite{Ar20+}). Thus, the question of which
topological Hopf algebras are holomorphically reflexive arises.

The holomorphic duality theories that we are discussing can be described with the help of \emph{reflexivity
diagrams}  of the form
\begin{equation*}
  \xymatrix{
H \ar@{|->}[d]&&(H^\bullet)'\ar@{|->}[ll]_{env}\\
H' \ar@{|->}[rr]_{env}&&H^\bullet\,,\ar@{|->}[u]
 }
\end{equation*}
where $H$ is an object of some category of topological Hopf algebras, the vertical arrows correspond
to some duality functor for topological vector spaces, and the horizontal ones correspond to some
envelope functor (left adjoint to some forgetful functor between categories of topological algebras).  If the
Hopf algebra structure is preserved under all these operations, it can be said that $H^\bullet$ is \emph{dual} to~$H$. If, in addition, the diagram is commutative (up to functor isomorphism), i.e.,  $H^{\bullet\bullet}\to H$ is a functor
isomorphism, then we can say what $H$ is \emph{reflexive}.

We take for $H$ the following topological Hopf algebras: the function algebras $\mathbb{C}^\Gamma$ and $\mathcal{O}(G)$ and also the group algebra  $\mathbb{C}\Gamma$ and the convolution algebra $\mathcal{O}(G)'$, where $\Gamma$ is a locally finite group and $G$ is a profinite group (here $\mathcal{O}(G)$ stands for the set of locally constant functions on $G$; we call them ``holomorphic,''  see Lemma~\ref{holloco}, and the prime denotes the strong dual space). In this paper, we prove the reflexivity of all these Hopf algebras within the framework of the approach of~\cite{Ar20+}.

The main difficulty in the scheme described above is the requirement that both the operations involved
in the definition of $H^\bullet$ must be functors between some categories of topological Hopf algebras. In this
paper, as well as in~\cite{Ar20+}, we consider Hopf algebras in the symmetric monoidal category of complete locally
convex spaces with the bifunctor $(-)\mathbin{\widehat{\otimes}} (-)$  of complete projective tensor product (Hopf $\mathbin{\widehat{\otimes}}$-algebras)  and
the functor of strong dual to a locally convex space. At the same time, in \cite{Ak08},  the spaces, their duals, and
tensor products are considered in the framework of the theory of stereotype spaces. The other approach
to definitions chosen in~\cite{Ar20+} is motivated by the fact that Akbarov's construction is rather complicated
technically and can be simplified under appropriate conditions on the topology of the underlying spaces.
Moreover, both approaches give the same result for connected Lie groups considered in~\cite{Ar20+}; see, e.g., \cite[Remark 2.9]{Ar20a}  (this is apparently not the case in the general situation).

Consider first the simplest case in which $G$ is a compactly generated Abelian Lie group. The
statement about the \emph{holomorphic reflexivity} can be described as the commutativity of the diagram

$$
  \xymatrix{
\mathcal{O}(G) \ar@{|->}[d]&&\mathcal{O}(G^\bullet)'\ar@{|->}[ll]_{Fourier}\\
\mathcal{O}(G)'  \ar@{|->}[rr]_{Fourier}&&\mathcal{O}(G^\bullet)\,,\ar@{|->}[u]
 }
$$
where $\mathcal{O}(G)$  is the set of holomorphic functions on $G$, the horizontal arrows are the Fourier transform,
and $G^\bullet$  stands for the group of holomorphic characters of $G$ (see Sec.\,\ref{absect}). A~remarkable observation made
in \cite{Ak08} is that, in this case, the Fourier transform coincides with the Arens--Michael envelope functor (the
completion with respect to the topology given by all continuous submultiplicative prenorms). Thus, we
can write the diagram in a form suitable not only for Abelian groups:
$$
  \xymatrix{
\mathcal{O}(G) \ar@{|->}[d]&&\widehat{{\mathscr A}}(G)'\ar@{|->}[ll]_{AM\,env}\\
{\mathscr A}(G) \ar@{|->}[rr]_{AM\,env}&&\widehat{{\mathscr A}}(G)\,,\ar@{|->}[u]
 }
$$
where ${\mathscr A}(G)=\mathcal{O}(G)'$, and the ``hat''
denotes the Arens--Michael envelope. Studying the commutativity
question for this diagram, in other words, the question of the \emph{holomorphic reflexivity} of $\mathcal{O}(G)$  in the
case of compactly generated Stein group (not necessarily Abelian) was started in \cite{Ak08}. For connected $G$,
the following answer was obtained in \cite{Ar20+}: \emph{$\mathcal{O}(G)$ is holomorphically reflexive if and only if $G$ is linear}.

Note that, for the class of Hopf $\mathbin{\widehat{\otimes}}$-algebras and in the general situation, there are no problems
when choosing the Arens--Michael envelope as the second functor in the definition, since, by \cite[Proposition~6.7]{Pir_stbflat},  we cannot leave this class by using the envelope functor.

In the present paper, we discuss the case of a zero-dimensional, i.e., discrete group~$G$. The algebra $\mathcal{O}(G)$ can be identified with the algebra $\mathbb{C}^G$  of all functions, and the reflexivity diagram takes the form
$$
  \xymatrix{
\mathbb{C}^G \ar@{|->}[d]&&(\widehat{\mathbb{C} G})'\ar@{|->}[ll]_{AM\,env}\\
\mathbb{C} G \ar@{|->}[rr]_{AM\,env}&&\widehat{\mathbb{C} G}\,.\ar@{|->}[u]
 }
$$
If $G$ is finitely generated (equivalently, compactly generated), then $\mathbb{C}^G$ is holomorphically reflexive (the
assertion is implicitly contained in~\cite{Ak08}). This is connected with the fact that, in this case, both the
algebras $\mathbb{C}^G$ and $\widehat{\mathbb{C} G}$ are nuclear Fr\'echet spaces; this is a situation in which the functor $H\mapsto H'$ is
compatible with the structure of Hopf $\mathbin{\widehat{\otimes}}$-algebra. It was discovered in  \cite{AHHFG} that $\widehat{\mathbb{C} G}$  need not be a Fr\'echet
space; moreover, the condition that G is finitely generated is necessary and sufficient. There are grounds
to assume that, in the general case, the topology on $\widehat{\mathbb{C} G}$  may have pathologies that can violate the
compatibility of the functor  $H\mapsto H'$ with the projective tensor product. Since we are interested in
arbitrary countable groups, it is important to understand how our duality scheme works for $H$ that are
not Fr\'echet spaces.

There are at least two possible ways of generalization here. The first way is in the modification
of the scheme itself within the framework of the theory of stereotype spaces (as in~\cite{Ak08}).  This version was developed in~\cite{Ak20+}; see below. In this paper, an alternative approach is chosen; we consider a wider
class of Hopf $\mathbin{\widehat{\otimes}}$-algebras for which the scheme works practically without modifications. This uses the fact that a Hopf $\mathbin{\widehat{\otimes}}$-algebra whose underlying space is a nuclear Fr\'echet space belongs to the class of
\emph{well-behaved} Hopf $\mathbin{\widehat{\otimes}}$-algebras, which, by definition, also includes the algebras with the underlying
nuclear (DF)-space. This concept was introduced in \cite{BFGP}; see the details in Sec.\,\ref{nonab} and a more detailed
discussion in~\cite{Ar20+}. The ground for distinguishing this class is the following property: if $H$ is a well-behaved
Hopf $\mathbin{\widehat{\otimes}}$-algebra, then $H'$ is also a well-behaved Hopf $\mathbin{\widehat{\otimes}}$-algebra (with respect to operations defined in
a dual way). Thus, the general scheme of~\cite{Ar20+}  admits an extension to the class of underlying spaces
opposite in properties to Fr\'echet spaces, namely, to (DF)-spaces. Moreover, as is proved in \cite{Ar21+},  the
Arens--Michael envelopes of group algebras of locally finite groups (which are opposite in properties to
finitely generated ones) are (DF)-spaces. This justifies the interest in locally finite groups in this context.
On the other hand, as is known, in the framework of the classical Pontryagin theory, the group dual to
a locally finite Abelian group is profinite (see, e.g., \cite[Theorem 2.9.6]{RZ10}). Therefore, the natural next step
is to include also profinite groups into consideration.

In the opinion of the author, it is expedient to consider the Abelian case separately, in view of the
fact that this part of the theory can be developed without involving Hopf algebras. For this reason,
we consider the holomorphic duality between locally finite and profinite groups in the Abelian case in
terms of characters and verify that it has exactly the same form as in the classical case in Sec.\,\ref{absect}. The
main results are contained in Sec.\,\ref{nonab}. We prove the holomorphic duality of function algebras and Hopf
group algebras for the same classes of groups, but already without the assumption of commutativity
(Theorems~\ref{lf} and~\ref{prf}). Note that it is reasonable to regard both ${\mathbb{C} \Gamma}$ and $\mathcal{O}(G)$ (where $\Gamma$ is locally
finite and G is profinite) as Hopf algebras of the same type and similarly combine $\mathbb{C}^\Gamma$ and  $\mathcal{O}(G)'$.

Note that, in \cite{Ak20+},  Akbarov proposed a generalization of his theory suitable, in particular, for group
algebras and function algebras on arbitrary (not only finitely generated, as before, or locally finite, as
here) countable discrete groups by showing that these algebras are holomorphically reflexive in this
broader sense (within the framework of the theory of stereotype spaces).   However, the result was
achieved at the price of further complication of the definitions. In particular, along with Arens--Michael
envelopes, two additional operations are applied: the pseudo-saturation, which, in general, modifies the
topology of the locally convex space, and the passage to an immediate stereotype subspace (see \cite{Ak20+}, formulas (100)--(105), and also~(81) and~(86)).  For function algebras and group algebras associated
with discrete groups, the second operation is excessive \cite[Theorems 4.3 and~4.4]{Ak20+}, but the first one (or
a similar one) is apparently inevitable. It is possible that one of the main results of the present paper,
Theorem~\ref{lf},  can be derived as a special case of \cite[Theorem 4.8]{Ak20+}.  However, the Hopf algebras treated
here have good topological properties, and both the additional operations are not necessary. Moreover,
the direct proof of Theorem~\ref{lf} is undoubtedly shorter. Therefore, we use a simpler definition from \cite{Ar20+}, to which the author of the present paper prefers to refer when possible. Establishing the limits of its
applicability is one of the goals of further research.

\section{Abelian case}
\label{absect}

In this section, we show that the Abelian version of holomorphic duality can be extended to locally
finite and profinite groups.

By a \emph{holomorphic character} of a complex Lie group $G$ we mean a holomorphic homo\-morphism
from $G$ to~$\mathbb{C}^\times$ (the group of invertible complex numbers) \cite[Sec.\,7.1]{Ak08}. Just as this is done in Pontryagin's theory, one can define the \emph{holomorphically dual group} $G^\bullet$ as the set of all holomorphic characters of $G$ endowed with the pointwise multiplication and the topology of uniform convergence on compacta. It was
Akbarov who showed that, if $G$ is a compactly generated Stein group, then $G^\bullet$ can be endowed with the
structure of a complex Lie group and is also a compactly generated Stein group (obviously Abelian) \cite[Theorem 7.1]{Ak08}. The direct proof uses structural theory, but does not provide a canonical construction.
However, the passage to Hopf algebras shows that such a construction does exist, see Remark~\ref{functsp}.

Let us turn to the case in which $G$ is zero-dimensional, i.e., discrete, and, in particular, every character
is holomorphic. Certainly, the term ``holomorphic'' in the context of discrete groups seems excessive,
but it is justified by the fact that the dual group to discrete may turn out to be a nondiscrete Lie group
(for example, $\mathbb{Z}^\bullet\cong\mathbb{C}^\times$) or even a pro-Lie group; see Remark~\ref{motpro}.

The next assertion is a particular case of the aforementioned theorem, Theorem~7.1 of~\cite{Ak08}.

\begin{pr}\label{fga}
Let $G$ be a finitely generated (discrete) Abelian group. Then its dual group $G^\bullet$ is a compactly generated Abelian Stein group. Here $G$ and $G^\bullet$
are holomorphically dual to each
other.
\end{pr}

We now turn to locally finite Abelian groups.

\begin{pr}\label{ablf}
Let $G$ be a locally finite Abelian group. Then  $G^\bullet$ coincides as a topological group
with the Pontryagin dual and is profinite.
\end{pr}
\begin{proof}
It can readily be seen that $G$ is locally finite if and only if it is periodic. In this case, all characters
of $G$ take values in~$\mathbb{T}$. Thus,  $G^\bullet$  coincides with the Pontryagin dual group, which is profinite in the case under consideration (see, e.g.,  \cite[Theorem 2.9.6]{RZ10}).
\end{proof}

Certainly, a profinite group need not be a Lie group. Therefore, before introducing the group holomorphically dual to it, it is necessary to clarify which functions on it are regarded as holomorphic ones.
To do this, we expand the class of groups under consideration. We call a topological group a \emph{complex
pro-Lie group} if it is a limit (in the category of topological groups) of a projective system in the category
of finite-dimensional complex Lie groups (cf. the real case in \cite{HM07,HM15}). In particular, profinite groups
are complex pro-Lie groups. Also, a complex-valued function on a complex pro-Lie group is said to be \emph{holomorphic} if it is a composition of a continuous homomorphism into a finite-dimensional complex
Lie group and a holomorphic function. Accepting this definition, we can consider the holomorphically
dual group~$G^\bullet$ for a profinite Abelian group $G$.

\begin{rems}\label{motpro}
(A)~The choice of the class of complex pro-Lie groups can be motivated by the following
observation. Just like every Abelian discrete group is the union of the inductive system of its finitely
generated subgroups, every Abelian finite-dimensional complex Lie group is the union of the inductive
system of its compactly generated open subgroups. In turn, to latter system, there corresponds projective
system of groups of holomorphic characters each of which is also an (Abelian finite-dimensional
compactly generated) complex Lie group. For example, the holomorphic characters of $\mathbb{Q}$  form a group
which is the complexification of a solenoid, namely, the projective limit of the coverings of the form
$\mathbb{C}^\times\to\mathbb{C}^\times$.
As we see, when considering the duality, the complex pro-Lie groups arise naturally.

(B)~The differences from the case of real pro-Lie groups are as follows. First, a continuous
homomorphism of complex Lie groups need not be holomorphic, while continuous homo\-morphisms
of real Lie groups are always smooth, which reduces real pro-Lie groups to projective limits in the
category of topological groups. This means that there should be noticeably fewer complex pro-Lie
groups than real ones, which is confirmed by the famous Yamabe theorem that every almost connected
locally compact group is a real pro-Lie group (see the discussion in~\cite{HM07}).
\end{rems}

\begin{lm}\label{holloco}
A function on a profinite group (not necessarily Abelian) is holomorphic if and only
if it is locally constant.
\end{lm}
\begin{proof}
We claim first that it follows from the definition we have adopted that the function on a profinite
group $G$ is holomorphic if and only if it is a composition of a continuous homomorphism into a finite
group and of some function. Indeed, let $\phi$ be a continuous homomorphism of $G$ into a Lie group. Since
$G$ is profinite, it has a neighborhood base of identity consisting of open subgroups \cite[Lemma 2.1.1]{RZ10}; however, the kernel of $\phi$ is open, because any Lie group does not contain small subgroups. Since $G$ is
compact, it follows that the image of $\phi$ is finite (for details, see, e.g., \cite{Ch}).

Thus, a function $f$ on $G$ is holomorphic if and only if there is a normal subgroup $H$ of finite index such
that $f$ is constant on every coset of $H$ (which is open according to \cite[Lemma 2.1.1]{RZ10}). On the other hand,
the last condition is equivalent to the fact that $f$ is locally constant (see, e.g.,  \cite[Lemma I.6.2]{Ca95}).
\end{proof}

Note that our definition of a holomorphic function on a profinite group is a particular case of the
definition of Bruhat-smooth functions formulated in \cite[Sec.\,1]{Br61}. Therefore, in the theory of profinite
groups and spaces, locally constant functions are sometimes called ``smooth''; see, e.g.,  \cite[Chapter I, Part~II, Sec.\,5]{Ca95}.

\begin{lm}\label{lincom}
 Every locally constant function on a profinite Abelian group $G$ is a linear combination of holomorphic characters.
\end{lm}
\begin{proof}
Let $f$ be a locally constant function. It follows from \cite[Theorem 2.1.3]{RZ10} that the topology has
a base consisting of cosets of subgroups of finite index. Since the group is compact, there is a finite
cover by subsets in this base on each of which $f$ is constant. Then $f$ is a linear combination of the
characteristic functions of these subsets.

It can readily be seen that every complex-valued function on a finite Abelian group is a linear
combination of characters of the group. Therefore, the function on $G$ which is constant on every coset of
some subgroup of finite index is a linear combination of holomorphic characters of $G$. In particular, this
concerns the characteristic function of the coset. The rest is clear.
\end{proof}

\begin{pr}\label{abprf}
Let $G$ be a profinite Abelian group. Then $G^\bullet$ coincides with the Pontryagin dual
group and is a locally finite discrete group.
\end{pr}
\begin{proof}
 By Lemma~\ref{holloco}, the holomorphic characters of $G$ are exactly its locally constant characters.
Since every locally constant function is continuous, we see that any holomorphic character of $G$ is a
continuous homomorphism into~$\mathbb{C}^\times$. Since the profinite group is compact, it follows that the image of
any continuous homomorphism from the group to  $\mathbb{C}^\times$  is contained in~$\mathbb{T}$.  On the other hand, the image
of a continuous homomorphism in~$\mathbb{T}$ is a finite group \cite[Lemma 2.9.2]{RZ10}, and thus this homomorphism
is locally constant  \cite[Lemma I.6.2]{Ca95}.
Thus, the set of holomorphic characters of $G$ coincides with the set
of continuous characters in the classical sense. Consequently,  $G^\bullet$ coincides with the Pontryagin dual
group. The latter is locally finite by \cite[Theorem 2.9.6]{RZ10}.
\end{proof}

Obviously, for every $g\in G$, the formula $\iota(g)\!:\chi\mapsto \chi(g)$, where $\chi\in G^\bullet$, defines a homomorphism $ G^\bullet\to \mathbb{C}^\times$. It follows from Propositions~\ref{ablf} and~\ref{abprf} that this character is holomorphic in both the cases,
and thus we obtain the following assertion concerning the reflexivity.

\begin{co}\label{abtref}
If an Abelian group $G$ is locally finite or profinite, then the homomorphism $\iota\!:G\to G^{\bullet\bullet}$ is well defined and is a topological isomorphism.
\end{co}

\begin{rem}
Note that here we do not claim that the isomorphism in Corollary~\ref{abtref}  is canonical, since
the specific definition of holomorphic functions on a profinite group is used. We still have no definition
suitable for both classes of groups (not to mention the consistency with the definition for Lie groups).
The search for such a definition is an urgent problem, and it seems that complex pro-Lie groups can be
useful here. A possible approach is to consider the duality at the level of Hopf algebras as the primary
concept; see Remark~\ref{functsp}.
\end{rem}

\section{General case}
\label{nonab}

The construction of holomorphic reflexivity for some topological Hopf algebras and the term itself
were introduced by Akbarov in the context of rigid stereotype Hopf algebras \cite[Russian p.\,128]{Ak08}. Akbarov
studied the holomorphic reflexivity for Hopf algebras associated with compactly generated complex Lie
groups as the first problem. Below we follow the formulation of the notion of holomorphic reflexivity given
in~\cite{Ar20+}. This narrower definition, although restricting the class of objects under consideration, allows one
to solve the problem without any reference to the theory of stereotype spaces, which is used in~\cite{Ak08}. (The
definition of holomorphic reflexivity in a broader sense, within the framework of this theory, can be found
in~\cite{Ak20+}.)

Recall that a Hopf algebra in the symmetric monoidal category of complete locally convex spaces
with the bifunctor $(-)\mathbin{\widehat{\otimes}} (-)$  of complete projective tensor product is called a \emph{Hopf $\mathbin{\widehat{\otimes}}$-algebra}, see \cite{Lit} or \cite{Pir_stbflat}.
Following \cite{BFGP}, we say that a Hopf $\mathbin{\widehat{\otimes}}$-algebra is
\emph{well-behaved}  if it is either a nuclear Fr\'echet space
or a (complete) nuclear (DF)-space.

The following definitions are formulated in \cite{Ar20+} for the case in which $H$ is a nuclear Fr\'echet space.
However, they can be carried over unchanged to the case in which $H$ is a nuclear (DF)--space, and hence
also to the case of an arbitrary well-behaved $H$. Recall that the completion of a topological associative $\mathbb{C}$-algebra $A$ with respect to the topology given by all continuous submultiplicative prenorms is called
the Arens--Michael envelope and is denoted by $\widehat A$. If $H$ is well-behaved, then $H^\bullet\!:=(H')\sphat\,\,$ is also a Hopf $\mathbin{\widehat{\otimes}}$-algebra by \cite[Proposition 6.7]{Pir_stbflat}  (not necessarily well-behaved) and an Arens--Michael algebra. We say that  $H^\bullet$ \emph{holomorphically dual} to~$H$ (cf. \cite[Definition~1.1]{Ar20+}). If $H^\bullet$ is nevertheless well-behaved, then the homomorphism of Hopf $\mathbin{\widehat{\otimes}}$-algebras  $H^{\bullet\bullet}\to H$ is well defined, cf. \cite[formula~(1.4)]{Ar20+}.  If,
moreover, this is a topological isomorphism, then we say that $H$ is  \emph{holomorphically reflexive} (cf. \cite[Definition~1.2]{Ar20+}).

Let $G$ be a countable discrete group. We are interested in two Hopf $\mathbin{\widehat{\otimes}}$-algebras associated with it,
namely, $\widehat{\mathbb{C} G}$ (the Arens--Michael envelope of the group algebra) and $\mathbb{C}^G$ (the algebra of all functions
on $G$). Note that $\widehat{\mathbb{C} G}$  is a nuclear space \cite[Theorem~2]{Ar21+}.
Moreover, $\widehat{\mathbb{C} G}$  is a Fr\'echet space if and only if $G$
is finitely generated \cite[Proposition~3.15]{AHHFG}  and, on the other hand, $\widehat{\mathbb{C} G}$  is a (DF)--space if and only if $G$ is
locally finite \cite[Theorem~6]{Ar21+}.  Thus, these two cases exhaust all opportunities for $\widehat{\mathbb{C} G}$ to be well-behaved.
On the contrary, $\mathbb{C}^G$  is always a nuclear Fr\'echet space, and thus a well-behaved Hopf $\mathbin{\widehat{\otimes}}$--algebra.

The following assertion is implicitly contained in \cite{Ak08}.
\begin{thm}\label{fg}
 Let $G$ be a finitely generated group. Then $\widehat{\mathbb{C} G}$ and $\mathbb{C}^G$ are holomorphically dual to
each other, and thus are holomorphically reflexive.
\end{thm}
\begin{proof}
We see directly from the definition that $\widehat{\mathbb{C} G}$ is holomorphically dual to~$\mathbb{C}^G$. The other assertion
follows from the fact that $\bigl(\widehat{\mathbb{C} G}\bigr)'\to \mathbb{C}^G$ is the Arens–Michael envelope, see \cite[Proposition~3.7]{Ar20+}.
\end{proof}

We now prove the first of our main results, namely, claiming that an analogous assertion holds also
for locally finite groups.

\begin{thm}\label{lf}
Let $G$ be a locally finite countable group. Then the Hopf $\mathbin{\widehat{\otimes}}$-algebras $\widehat{\mathbb{C} G}$ and $\mathbb{C}^G$ are holomorphically dual to each other and thus are holomorphically reflexive.
\end{thm}
(Since the topological properties of Hopf $\mathbin{\widehat{\otimes}}$-algebras are substantial in the definition of reflexivity, here we have to additionally assume that the group is countable, in contrast to the Abelian case, which
was described in terms of characters.)

\begin{proof}
Since $G$ is locally finite, the homomorphism $\mathbb{C} G\to\widehat{\mathbb{C} G}$  is a topological isomorphism according to \cite[Proposition~5]{Ar21+} (Here $\mathbb{C} G$  is equipped with the strongest locally convex topology). Then $ \bigl(\widehat{\mathbb{C} G}\bigr)'\cong (\mathbb{C} G)'\cong  \mathbb{C}^G$  (as locally convex spaces). Since $\mathbb{C}^G$ is an Arens--Michael algebra, it follows
that $(\widehat{\mathbb{C} G})^\bullet\cong \mathbb{C}^G$ (as a Hopf $\mathbin{\widehat{\otimes}}$-algebra). Similarly, it follows from $(\mathbb{C}^G)'\cong \mathbb{C} G$ that $(\mathbb{C}^G)^\bullet\cong \widehat{\mathbb{C} G}$.
\end{proof}

Thus, if $G$ is a locally finite countable group, then the reflexivity diagram becomes
$$
  \xymatrix{
\mathbb{C}^G \ar@{|->}[d]&&\mathbb{C}^G \ar@{|->}[ll]_{AM\,env}\\
\mathbb{C} G \ar@{|->}[rr]_{AM\,env}&&\mathbb{C} G \,.\ar@{|->}[u]
 }
$$

Since the locally finite and profinite groups are dual to each other in the Abelian case, it is natural to
assume that an assertion similar to Theorem~\ref{lf} holds also for profinite groups (not necessarily Abelian).
This is really the case, but auxiliary results are needed for the verification.

Let $G$ be a profinite group. As above, denote by $\mathcal{O}(G)$  the set of holomorphic functions on $G$. Since $\mathcal{O}(G)$ is an algebra of locally constant functions by Lemma~\ref{holloco}, it follows that the topology on this
algebra can be introduced in the standard way (see, e.g.,  \cite[Sec.\,5]{Pl19}). Namely, let $I$ be the family
of all possible partitions of $G$ into finitely many disjoint open subsets. For a given $\alpha\in I$, denote by $\mathcal{O}(G)_\alpha$ the algebra of functions that are constant on each of the open subsets of~$\alpha$. Since $\mathcal{O}(G)_\alpha$ is
finite dimensional, it can be regarded with the standard topology. Then $\mathcal{O}(G)$ is the union of the directed
family $(\mathcal{O}(G)_\alpha)$  and it can be equipped with the topology of the inductive limit. In addition, one can give a more sparing description of the topology by presenting $G$ as the projective limit of a directed family of
finite groups $(G_\alpha)$. Then $\mathcal{O}(G)_\alpha\cong \mathbb{C}^{G_\alpha}$  for every~$\alpha$, and $\mathcal{O}(G)$  is the inductive limit of the sequence of
finite-dimensional linear spaces $(\mathbb{C}^{G_\alpha})$.

By \cite[Theorem 7.2]{Ak08}, $\mathcal{O}(G^\bullet)\cong \mathcal{O}(G)^\bullet$  for every compactly generated Abelian Stein group $G$. We now prove a similar result concerning the consistency of the notions of holomorphic reflexivity for the groups and Hopf algebras in our case. First, Lemmas~\ref{holloco} and~\ref{lincom} immediately imply the following assertion.

\begin{lm}\label{proGbu}
Let $G$ be a profinite Abelian group. Then  $\mathcal{O}(G)\cong \mathbb{C}(G^\bullet)$.
\end{lm}

To provide the desired topological properties of Hopf $\mathbin{\widehat{\otimes}}$-algebras, we also impose the countability condition below. Just like a locally finite group is countable if and only if it is a countable inductive limit of finite groups, a profinite group has a countable topology base if and only if it is a countable projective
limit of finite groups (this is easy to show; cf., e.g., \cite[p.\,11, Corollary 1.1.13]{RZ10}). Moreover, a locally
finite Abelian group is countable if and only if its holomorphically dual profinite group has a countable
topology base (this follows from  \cite[p.\,60, Lemma 2.9.3]{RZ10}).
Below, for the sake of unity of notation, we
also sometimes denote the set $\mathbb{C}^G$  of all functions on a discrete group $G$ by $\mathcal{O}(G)$.

\begin{pr}\label{TrF}
Let $G$ be a locally finite countable Abelian group or a profinite Abelian group
with countable base. Then $\mathcal{O}(G^\bullet)\cong \mathcal{O}(G)^\bullet$.
\end{pr}
\begin{proof}
Assume that G is locally finite and countable. Then $G^\bullet$ is profinite according to Proposition~\ref{ablf}.
Lemma~\ref{proGbu} implies that $\mathcal{O}(G^\bullet)$ coincides with $\mathbb{C} (G^{\bullet\bullet})$. It follows from Corollary~\ref{abtref} that $ G^{\bullet\bullet}\cong G$, and
thus $\mathcal{O}(G^\bullet)\cong\mathbb{C} G$.
Since $G$ is discrete, we obtain $\mathcal{O}(G)=\mathbb{C}^G$, and, since $G$ is countable, it follows from
Theorem~\ref{lf} that $\mathbb{C} G\cong (\mathbb{C}^G)^\bullet=\mathcal{O}(G)^\bullet$. Thus, $\mathcal{O}(G^\bullet)\cong \mathcal{O}(G)^\bullet$.

Assume that $G$ is profinite and has a countable base. Then $G^\bullet$ is locally finite by Proposition~\ref{abprf},  and
thus $\mathcal{O}(G^\bullet)=\mathbb{C}^{G^\bullet}$. On the other hand, $\mathbb{C}(G^\bullet)\cong \mathcal{O}(G)$ by Lemma~\ref{proGbu}.  Since $G$ has a countable base,
it follows that $G^\bullet$  is countable. By Theorem~\ref{lf} we have $ (\mathbb{C}(G^\bullet))^\bullet \cong   \mathbb{C}^{G^\bullet}$.  Thus, $\mathcal{O}(G^\bullet)\cong \mathcal{O}(G)^\bullet$.
\end{proof}

The key point needed to be verified to prove an analogue of Theorem~\ref{lf}  for a profinite group $G$ is
that $\mathcal{O}(G)$ and $\mathcal{O}(G)'$ are Arens--Michael algebras, and thus the Arens--Michael envelope functor acts
trivially both on $\mathcal{O}(G)$ and on $\mathcal{O}(G)'$.

\begin{lm}\label{OGstrt}
 Let G be a profinite group with countable base. Then the topology on $\mathcal{O}(G)$ coincides with the strongest locally convex topology.
\end{lm}
\begin{proof}
Presenting $G$ as the projective limit of a sequence $(G_n)$ of finite groups, we see that $\mathcal{O}(G)$  is
the inductive limit of a sequence  $(\mathbb{C}^{G_n})$ of finite-dimensional linear spaces,  which can be regarded as
subspaces of $\mathcal{O}(G)$.  Choose in $\mathcal{O}(G)$  a countable linear basis consistent with the sequence $(\mathbb{C}^{G_n})$  and take the corresponding decomposition of   $\mathcal{O}(G)$ into a direct sum of one-dimensional subspaces. Then the inductive topology on  $\mathcal{O}(G)$  coincides with the direct sum topology, which, in turn, is the strongest locally convex one; see \cite{Sc71}, Sec.\,II.6, Russian pp.\,74--75, (6.2)  and an example after it.
\end{proof}

Note that the question of when an algebra endowed with the inductive topology has continuous
multiplication and, in particular, is an Arens--Michael algebra turned out to be rather subtle (see a
discussion in \cite{HW04}).  However, there are no problems in our special case, namely, the following assertion
holds.

\begin{pr}\label{OGAM}
 Let $G$ be a profinite group with  countable base. Then $\mathcal{O}(G)$ is an Arens--Michael
algebra
\end{pr}
\begin{proof}
By  \cite[Th{e}or{e}m~2.1]{AN96}, the inductive limit topology (in the category of locally convex spaces) in
the case of a sequence of normed algebras can be given by a family of submultiplicative prenorms. To
ensure that an Arens--Michael algebra is obtained, it remains to verify the completeness. By Lemma~\ref{OGstrt},  the topology on $\mathcal{O}(G)$ is the strongest locally convex one, whence the completeness follows.
\end{proof}

We now show that $\mathcal{O}(G)$  is a Hopf $\mathbin{\widehat{\otimes}}$-algebra.

\begin{lm}\label{prfOO}
Let $G$ be a profinite group with countable base. Then $\mathcal{O}(G)\otimes \mathcal{O}(G) \cong \mathcal{O}(G\times G)$  as a linear space, and $\mathcal{O}(G)\mathbin{\widehat{\otimes}} \mathcal{O}(G) \cong \mathcal{O}(G\times G) $ как as a locally convex space.
\end{lm}
\begin{proof}
We use Lemma~\ref{holloco}.
First, we claim that the mapping $\psi\!:\mathcal{O}(G)\otimes \mathcal{O}(G) \to\mathcal{O}(G\times G) $  is well
defined. Indeed, any element of the tensor product is a finite sum of elementary tensors $f_i\otimes h_i$.
Since all $f_i$ and $h_i$  are locally constant, it follows that the function on $G\times G$  corresponding to this sum is locally
constant.

Note that $\psi$ is injective (this follows from the fact that the analogous mapping for functions on finite
sets is injective). On the other hand, represent $G$ as the projective limit of a sequence of finite groups
 $(G_n)$. Then $G\times G$ is the projective limit of the sequence  $(G_n\times G_n)$. Therefore, every locally constant
function on $G\times G$  is constant on cosets of the kernel of the homomorphism  $G\times G\to G_n\times G_n$  for
some $n$. This readily implies that $\psi$  is surjective. Thus, $\psi$  is an isomorphism of linear spaces.

By Lemma~\ref{OGstrt}, the topology on the space $\mathcal{O}(G)$ coincides with the strongest locally convex topology,
which means that this space can be represented as a direct sum of one-dimensional subspaces. Further,
we use the fact that there is an isomorphism
$$
\mathcal{O}(G)\mathbin{\widehat{\otimes}} \mathcal{O}(G)\cong (\mathcal{O}(G)'\mathbin{\widehat{\otimes}} \mathcal{O}(G)')'.
$$
Since the strong dual of a sum is the product and conversely \cite[Sec.\,22.5, p.\,287]{Kot1}, and the
projective tensor product commutes with direct products  \cite[Sec.\,41.6, p.\,194, (5)]{Kot2}, it follows that $\mathcal{O}(G)\mathbin{\widehat{\otimes}} \mathcal{O}(G)$  is also a direct sum of one-dimensional subspaces, and thus its topology is the strongest
locally convex one. Since $\mathcal{O}(G\times G) $ is also equipped with the strongest locally convex topology, we have
a topological isomorphism
$\mathcal{O}(G)\mathbin{\widehat{\otimes}} \mathcal{O}(G) \cong \mathcal{O}(G\times G) $.
\end{proof}

\begin{pr}\label{OGHopf}
Let $G$ be a profinite group with countable base. Then $\mathcal{O}(G)$  is a commutative
well-behaved Hopf $\mathbin{\widehat{\otimes}}$-algebra with respect to standard operations for a function algebra on a
group.
\end{pr}
\begin{proof}
It follows from Proposition~\ref{OGAM} that $\mathcal{O}(G)$ is a $\mathbin{\widehat{\otimes}}$-algebra. By Lemma~\ref{prfOO}, the comultiplication is
well defined. Since $\mathcal{O}(G)$  is the set of locally constant functions (Lemma~\ref{holloco}),  it is invariant with respect to the antipode. The continuity of the operations follows from the fact that $\mathcal{O}(G)$ is regarded with the strongest locally convex topology (Lemma~\ref{OGstrt}).  The axioms of Hopf $\mathbin{\widehat{\otimes}}$-algebra are verified directly. Since the topology is the strongest locally convex one, it follows that $\mathcal{O}(G)$ is a nuclear (DF)-space, which means that $\mathcal{O}(G)$ is a well-behaved Hopf $\mathbin{\widehat{\otimes}}$-algebra.
\end{proof}

\begin{pr}\label{OGprAM}
 Let $G$ be a profinite group with countable base. Then  $\mathcal{O}(G)'$  is an Arens--Michael
algebra.
\end{pr}
\begin{proof}
Just as in the proof of Lemma~\ref{OGstrt}, we represent $G$ as the projective limit of a sequence   $(G_n)$ of
finite groups and, correspondingly, represent $\mathcal{O}(G)$  as the inductive limit of the sequence  $(\mathbb{C}^{G_n})$ of
finite-dimensional linear spaces. It can readily be seen that the linking mappings are homomorphisms of Hopf $\mathbin{\widehat{\otimes}}$-algebras. In particular, the family $(\mathbb{C}^{G_n}\to \mathcal{O}(G))$ is a cone in the category of $\mathbin{\widehat{\otimes}}$-coalgebras. Applying the functor of strong dual, we obtain the cone  $(\mathcal{O}(G)'\to \mathbb{C} G_n)$ in the category
of $\mathbin{\widehat{\otimes}}$-algebras (taking into account that $(\mathbb{C}^{G_n})'\cong \mathbb{C} G_n$).  Being finite dimensional, $\mathbb{C} G_n$ can be regarded
as Banach algebras (for example, by providing them with $\ell_1$-norms). Denote by $A$ the Arens--Michael
algebra which is the projective limit of the sequence $(\mathbb{C} G_n)$.  The above cone induces a homomorphism of $\mathbin{\widehat{\otimes}}$-algebras $\mathcal{O}(G)'\to A$.  It remains to show that this homomorphism is an isomorphism of $\mathbin{\widehat{\otimes}}$-algebras.
It can readily be seen that, to this end, it suffices to establish that this homomorphism is a topological
isomorphism of locally convex spaces.

The inductive system $(\mathbb{C}^{G_n})$ is regular in the following sense: every bounded subset of $\mathcal{O}(G)$ is
bounded in some $(\mathbb{C}^{G_n})$ (see the definition in \cite[Sec.\,1, p.~46, Definition~2]{Bi86} or \cite[Sec.\,23, Sec.\,5.2, p.\,123]{FW68}).  Indeed, since all spaces are finite-dimensional, it follows that the system is strict, and one can apply \cite[Sec.\,24,  p.\,127, Satz\,2.2]{FW68}. The regularity implies that there is a topological isomorphism
between the strong dual to the inductive limit and the projective limit of strong duals \cite[Sec.\,2, p.\,57, Proposition~1]{Bi86} (see, e.g., a proof in \cite[Sec.\,25, p.\,145, Satz\,2.1]{FW68}).  Since the first space is $\mathcal{O}(G)'$ and
the other is $A$, this means that  $\mathcal{O}(G)'\to A$ is a topological isomorphism, as required.
\end{proof}

In conclusion, we prove our second main result, which is similar to Theorem~\ref{lf}.

\begin{thm}\label{prf}
Let $G$ be a profinite group with countable base. Then the $\mathbin{\widehat{\otimes}}$-algebras $\mathcal{O}(G)$ and $\mathcal{O}(G)'$  are holomorphically dual to each other and thus are holomorphically reflexive.
\end{thm}
\begin{proof}
Since $\mathcal{O}(G)'$ is an Arens--Michael algebra (according to Proposition~\ref{OGprAM}),  we see that $\mathcal{O}(G)^\bullet=\mathcal{O}(G)'$. Then $\mathcal{O}(G)'$ is a nuclear Fr\'echet space being the strong dual to a nuclear (DF)-space, and hence a well-behaved Hopf   $\mathbin{\widehat{\otimes}}$-algebra. Thus, the homomorphism of Hopf $\mathbin{\widehat{\otimes}}$-algebras
$$(\mathcal{O}(G)')^{\bullet}\to \mathcal{O}(G)$$
is well defined. Further, since any space with the strongest locally convex topology is reflexive in the
sense of the strong dual, it follows that  $\mathcal{O}(G)''\cong \mathcal{O}(G)$. Finally, $\mathcal{O}(G)$ is an Arens--Michael algebra by
Proposition~\ref{OGAM},  and thus $(\mathcal{O}(G)')^{\bullet}\to \mathcal{O}(G)$  is an isomorphism.
\end{proof}
Thus, if $G$ is a profinite group with countable base, then the reflexivity diagram becomes
$$
  \xymatrix{
\mathcal{O}(G) \ar@{|->}[d]&&\mathcal{O}(G)\ar@{|->}[ll]_{AM\,env}\\
\mathcal{O}(G)' \ar@{|->}[rr]_{AM\,env}&&\mathcal{O}(G)'\,.\ar@{|->}[u]
 }
$$

\begin{rem}\label{functsp}
If $G$ is an Abelian group, then $\widehat{\mathbb{C} G}$ is a commutative Arens–Michael algebra. Applying
the Gelfand transformation, we can identify this algebra with the algebra of functions on the spectrum,
i.e., on the set of continuous algebraic characters (the Gelfand transformation is injective, since $\widehat{\mathbb{C} G}$ is a dense subalgebra in the semisimple Banach algebra $L^1(G)$). It can readily be seen that the
resulting group is just~$G^\bullet$. If, in addition, G is finitely generated, then $\widehat{\mathbb{C} G}$ is a holomorphically finitely
generated Hopf algebra in the sense of \cite{AHHFG}.  It follows from the commutativity of $\widehat{\mathbb{C} G}$ that $G^\bullet$  has a canonical Lie group structure \cite[Theorem~2.2]{AHHFG}.  A similar assertion holds also in the more general case of compactly generated Lie groups \cite[Theorem~3.1]{AHHFG}, which allows one to give a canonical form to
Akbarov's construction in~\cite[теорема~7.1]{Ak08}.

If $G$ is Abelian and locally finite, then, in a similar way, the Gelfand transformation gives us not
only a group $G^\bullet$, but also the algebra of functions on $G^\bullet$.
Thus, transposing $G$ and $G^\bullet$,
we can use the
formula of Lemma~\ref{proGbu} as the definition of the algebra of holomorphic functions on a profinite (Abelian)
group and obtain the definition accepted above as its consequence (recall that $\widehat{\mathbb{C} G}\cong \mathbb{C} G$ according to
Proposition~\ref{OGAM}).
\end{rem}

The author is grateful to S.\,S.~Akbarov for useful remarks that improved the quality of the presentation.

\end{document}